\documentclass[12pt]{amsart}
\usepackage{amssymb,amscd,amsthm,amsmath,color}
\usepackage{fullpage}
\usepackage{epsfig}
\usepackage{graphicx}
\usepackage{xypic}
\usepackage{url}
\usepackage{mathtools}
\usepackage{hyperref}
\usepackage{mathrsfs}
\usepackage{quiver}

\numberwithin{equation}{section}

\newcommand{\codim}{\operatorname{codim}}

\newcommand{\NE}{\overline{\operatorname{NE}}}

\newcommand{\Mor}{\operatorname{Mor}}

\newtheorem{theorem}{Theorem}[section]
\newtheorem{lemma}[theorem]{Lemma}
\newtheorem{proposition}[theorem]{Proposition}

\theoremstyle{definition}

\definecolor{ao(english)}{rgb}{0.0, 0.5, 0.0}

\input xy
\xyoption{all}

\makeatletter
\def\dasharrowfill@#1#2#3#4{%
        $\m@th
        \thickmuskip0mu
        \medmuskip\thickmuskip
        \thinmuskip\thickmuskip
        \relax
        #4#1\mkern2mu
        \xleaders\hbox{$#4\mkern2mu#2\mkern2mu$}\hfill
        \mkern2mu
        #3$%
}

\def\dashrightarrowfill@{\dasharrowfill@\relbar\relbar\rightarrow}

\providecommand*\xdashrightarrow[2][]{%
  \ext@arrow 0055{\dashrightarrowfill@}{#1}{#2}}
\makeatother

\begin{document}

\title{Optimal bounds in Bend-and-Break}
%\author[]{}

\author{Eric Jovinelly}
\address{Department of Mathematics \\
Brown University \\
Box 1917 \\
151 Thayer Street \\
Providence, RI, 02912}
\email{eric\_jovinelly@brown.edu}

\author{Brian Lehmann}
\address{Department of Mathematics \\
Boston College  \\
Chestnut Hill, MA \, \, 02467}
\email{lehmannb@bc.edu}

\author{Eric Riedl}
\address{Department of Mathematics \\
University of Notre Dame  \\
255 Hurley Hall \\
Notre Dame, IN 46556}
\email{eriedl@nd.edu}

\begin{abstract}
We improve the Bend-and-Break result of Miyaoka and Mori by establishing the optimal degree bound.  Our result also yields optimal bounds on lengths of extremal rays of log canonical pairs.
\end{abstract}

\maketitle

\section{Introduction}

Mori's Bend-and-Break lemma is a fundamental tool for working with curves on projective varieties.  Different versions of this important result have been established by \cite{Mori79, Mori82, MM86, Kollar}. Our main goal is to strengthen \cite[Theorem 5]{MM86} and to apply it to lengths of extremal rays, answering questions posed by \cite{Kollar, Nikulin96, Matsuki02, Fujino11} (and others).

\begin{theorem} \label{theo:cleanbandb}
Let $X$ be a projective variety over an algebraically closed field of arbitrary characteristic.  Let $H$ be a nef $\mathbb{R}$-Cartier divisor on $X$.  Suppose there exists an irreducible curve $C \subset X$ contained in the smooth locus of $X$ such that
$$K_X \cdot C < 0.$$
Then for every closed point $x \in C$, there exists a rational curve $R$ containing $x$ such that
$$H \cdot R \leq (\dim X + 1) \frac{H \cdot C}{-K_X \cdot C}.$$
\end{theorem}

The constant $(\dim X + 1)$ in Theorem \ref{theo:cleanbandb} improves the constant $(2 \dim X)$ given in \cite[Theorem 5]{MM86}. 
 Our improvement is optimal as we may let $X$ be $\mathbb{P}^{n}$ and $H$ be a hyperplane.

The proof of \cite[Theorem 5]{MM86} uses the fact that a one-dimensional family of maps $C \to X$ with a fixed point must break off a rational curve.  Our key technical statement (``Bend-and-Shatter'', Lemma \ref{lem:genusGBendandBreak}) shows that a $k$-dimensional family of curves $C \to X$ that fixes $k$ points must break off $k$ rational curves.  When combined with the reduction steps of \cite[Theorem 5]{MM86} and \cite[II.5.8 Theorem]{Kollar}, we obtain a quick proof of Theorem \ref{theo:cleanbandb}.

\subsection{Extremal rays}
One of Mori's first applications for Bend-and-Break was the study of extremal rays of the pseudo-effective cone of curves.  \cite[Theorem 1.4]{Mori82} proved that for a smooth projective variety $X$ every $K_{X}$-negative extremal ray of the pseudo-effective cone contains a rational curve $C$ satisfying $-K_{X} \cdot C \leq \dim X + 1$.

For a klt pair $(X,\Delta)$ with $\mathbb{Q}$-coefficients, \cite{Kawamata91} proved an analogous statement with the upper bound $(2 \dim X)$.  This was extended to dlt pairs with $\mathbb{R}$-coefficients by Shokurov in the appendix to \cite{Nikulin96} and by \cite[Theorem 3.8.1]{BCHM10}.   Using Theorem \ref{theo:cleanbandb} in place of \cite[Theorem 5]{MM86} in these arguments, we obtain the optimal degree bound:

\begin{theorem} \label{theo:kltextremalrays}
Let $(X,\Delta)$ be a dlt pair over an algebraically closed field of characteristic $0$.  Suppose that $\pi: X \to Z$ is the contraction of a $(K_{X} + \Delta)$-negative extremal face $\mathcal{R}$ of $\NE(X)$.  For any positive-dimensional irreducible component $F$ of a fiber of $\pi$, there is a rational curve $C$ in $F$ satisfying:
\begin{enumerate}
\item The class of $C$ is contained in the face $\mathcal{R}$.
\item The deformations of $C$ sweep out $F$.
\item $-(K_{X} + \Delta) \cdot C \leq \dim F + 1$.
\end{enumerate}
If furthermore $(X,\Delta)$ is klt and $\pi$ is a birational contraction, then we can ensure a strict inequality in (3).
\end{theorem}

The arguments of \cite[Theorem 18.2]{Fujino11} extend this result to lc pairs with $\mathbb{R}$-coefficients.

\begin{theorem} \label{theo:introlengthray}
Let $(X,\Delta)$ be an lc pair over an algebraically closed field of characteristic $0$.  Then every $(K_{X}+\Delta)$-negative extremal ray of the pseudo-effective cone of curves is generated by a rational curve $C$ with $-(K_{X} + \Delta) \cdot C \leq \dim X+1$. 
\end{theorem}

Note that this length bound was known previously for toric varieties by \cite{Fujino03} and in the setting of LCIQ singularities by \cite{CT09}.

\vspace{.2in}

\noindent \textbf{Acknowledgements:}
We thank Izzet Coskun, Brendan Hassett, James M\textsuperscript{c}Kernan, Joaqu\'in Moraga, and Sho Tanimoto for helpful conversations.  We are grateful to Yuchen Liu for suggesting we think about lengths of extremal rays.

Eric Jovinelly was supported by an NSF postdoctoral research fellowship, DMS-2303335. Brian Lehmann was supported by Simons Foundation grant Award Number 851129.  Eric Riedl was supported by NSF CAREER grant DMS-1945944, and Simons Foundation grants 00011850 and 00013673.

\section{Breaking curves: low degree rational curves}

In this section, we establish Bend-and-Shatter and use it to prove Theorem \ref{theo:cleanbandb}.  We let $\overline{\mathcal{M}}_{g,n}(X)$ denote the Kontsevich moduli stack of stable maps and let $\mathcal{M}_{g,n}(X)$ denote the open substack of maps with smooth irreducible domain.

\begin{lemma}[Bend-and-Shatter] \label{lem:genusGBendandBreak}
Let $X$ be a projective variety over an algebraically closed field of arbitrary characteristic.  Fix a stable irreducible marked curve $(C, q_1, \dots, q_r) \in \mathcal{M}_{g,r}$.

For some $k \leq r$, let $p_1, \dots, p_k$ be points of $X$.  Suppose there exists a $k$-dimensional locally closed substack $S \subset \mathcal{M}_{g,r}(X,\beta)$ parametrizing pointed maps $s: (C, q_1, \ldots , q_r) \to X$ with $s(q_i) = p_i$ for all $i\leq k$.  Then there exists a stable map $s' : (C', q'_1, \dots, q'_r) \to X$ in the closure of $S$ in $\overline{\mathcal{M}}_{g,r}(X,\beta)$ such that 
\begin{itemize}
\item $s'(q'_i) = p_i$ for all $i\leq k$;
\item for each $i \leq k$ there is a tree of rational curves $C'_{i} \subset C'$ such that $q'_i \in C'_i$ and $s'$ does not contract $C_i'$ to a point; and 
\item the stabilization of $(C', q'_1, \dots, q'_r)$ is $(C, q_1, \dots, q_r)$.  In particular, the stabilization map $(C', q'_1, \dots, q'_r) \to (C, q_1, \dots, q_r)$ contracts $C'_i$ to $q_i \in C$.
\end{itemize}
\end{lemma}

\begin{proof}
Let $U$ be the preimage of $S$ in $\mathcal{M}_{g,r+k}(X, \beta)$ under the map $\pi: \overline{\mathcal{M}}_{g,r+k}(X, \beta) \to \overline{\mathcal{M}}_{g,r}(X, \beta)$ which forgets the last $k$ points. Thus, $U$ parametrizes maps in $S$ together with the choice of $k$ additional points $\{q_{r+i} \}_{i=1}^{k}$ on $C$.   We also let $\psi: \overline{\mathcal{M}}_{g,r+k}(X, \beta) \to \overline{\mathcal{M}}_{g,r+k}$ be the forgetful map and let $\phi : \overline{\mathcal{M}}_{g,r+k} \to \overline{\mathcal{M}}_{g,r}$ be the map forgetting the last $k$ markings.  By construction the closure of the image of $U$ under $\psi$ is the fiber $F$ of $\phi$ over $(C, q_1, \ldots , q_r)$.  Note that the non-empty fibers of $\psi: U \to F$ have dimension $k$.

Fix a very ample line bundle $\mathcal{L}$ on $X$.  Set $U_{0} = U$ and for $1 \leq i \leq k$ define $U_{i}$ inductively by choosing a general section $D_{i}$ of $\mathcal{L}$ and letting $U_{i} \subset U_{i-1}$ be the substack of maps $s$ such that $s(q_{r+i}) \in D_{i}$.  We prove by induction that the dimension of the fiber of $\psi|_{U_{i}}$ over a general point of $F$ is at least $k-i$.

By induction we know the fiber $V_{i-1}$ of $\psi|_{U_{i-1}}$ over a general point $(C,q_{1},\ldots,q_{r+k})$ of $F$ has positive dimension.  Note that the image of $V_{i-1}$ in $\mathcal{M}_{g,0}(X)$ does not depend on the choice of marked points $q_{r+i},\ldots,q_{r+k}$.  Furthermore, this image must have positive dimension since $V_{i-1}$ has positive dimension and parametrizes maps from a fixed marked curve.  Thus for a general point of $F$ the image $s(q_{r+i})$ sweeps out a locus of dimension $\geq 1$ in $X$ as we vary $s \in V_{i-1}$.   We conclude that the preimage of the general divisor $D_{i}$ under the evaluation map $ev_{r+i}|_{U_{i-1}}$ meets $V_{i-1}$.  In particular the dimension of the general fiber of $\psi|_{U_{i}}$ is at most one less than the dimension of the general fiber of $\psi|_{U_{i-1}}$, proving the claim.

Let $\overline{U_{k}}$ be the closure of $U_{k}$ in $\overline{\mathcal{M}}_{g,r+k}(X, \beta)$.  There is an element of $\overline{U_{k}}$ lying over the locus in $F$ where $q_{r+i}$ specializes to $q_i$ for each $i$. Let $s':(C', q'_1, \dots, q'_{r+k}) \to X$ be the corresponding stable map. Because $\pi(s')$ lies in the closure $\overline{S}$ of $S$ in $\overline{\mathcal{M}}_{g,r}(X,\beta)$, we know that the stabilization of $(C',q'_1,\dots,q'_r)$ must be $(C, q_1, \dots, q_r)$.  Thus each $q'_{i}$ is contained in a tree of rational curves (which also contains $q'_{r+i}$) that is contracted by the stabilization map. Likewise, because $\pi(s')$ lies in $\overline{S}$ we see that $s'(q'_i) = p_i$ for all $i \leq k$.  For all $i,j \leq k$, generality of $D_j$ ensures it is disjoint from $p_i$. Because $s'(q'_{r+i})$ must lie in $D_i$, we see that the tree of rational curves containing $q'_{i}$ has to map to a curve in $X$ connecting $p_i$ to $D_i$; in particular, some component is not contracted by $s'$.
\end{proof}

The next proposition relates the dimension of a family of curves to the number of rational curves that can be broken off using Lemma \ref{lem:genusGBendandBreak}.

\begin{proposition}\label{prop:totalbreaking}
Let $X$ be a projective variety and let $C$ be a smooth projective curve of genus $g$ over an algebraically closed field of arbitrary characteristic.  
Suppose $M \subset \Mor(C,X)$ is an irreducible locally closed subvariety.   Set $k = \lfloor \frac{\dim M}{\dim X+1} \rfloor$ and let $s : C \to X$ be any map parametrized by $M$.  If $2g - 2 + k > 0$, then the closure of the image of $M$ in $\overline{\mathcal{M}}_{g,0}(X,\beta)$ parametrizes a map $s': C' \to X$ satisfying:
\begin{itemize}
\item $C'$ consists of the union of $C$ with at least $k$ trees of rational curves, and
\item at least $k$ of these trees contain an irreducible component $T$ such that $s'$ realizes $T$ as a non-contracted rational curve on $X$ that passes through a general point of $s(C)$.
\end{itemize}
\end{proposition}

\begin{proof}
Let $q_1, \ldots , q_k \in C$ be $k$ general points in $C$.  Let $T_i = \{\tilde{s} \in M \mid \tilde{s}(q_i) = s(q_i)\}$. Since $s \in T_i$ for all $i$, the intersection $S := \cap_i T_i$ is nonempty.  Moreover, as $\codim(T_i, M) \leq \dim X$, we get $\codim(S, M) \leq k(\dim X)$.  Thus $\dim S \geq k$.  

Because $2g - 2 + k > 0$, the natural map $\pi : S \to \overline{\mathcal{M}}_{g,k}(X, \beta)$ is generically finite.  Apply Lemma \ref{lem:genusGBendandBreak} to $\pi(S)$ and let $s': (C', q_1', \ldots , q_k') \to X$ be the stable map it identifies.  The desired stable map is obtained from $s'$ by forgetting the $k$ marked points and stabilizing.
\end{proof}

We are now equipped to prove Theorem \ref{theo:cleanbandb} via a dimension counting argument.

\begin{proof}[Proof of Theorem \ref{theo:cleanbandb}:]
First suppose that our ground field is algebraically closed of characteristic $p > 0$ and that $H$ is $\mathbb{Q}$-Cartier.   After rescaling $H$ we may suppose it is Cartier.  We write $i: C' \to X$ for the normalization of $C$.  For $m > 0$, let $s_m : C' \to X$ be the precomposition of $i$ with the $m^{\text{th}}$ iterate of the Frobenius.  
The dimension $d_{m}$ of $\Mor(C',X)$ at $s_m$ satisfies
$$d_{m} \geq p^m(-K_X \cdot i_*C') - g\dim X,$$
where $g$ is the genus of $C'$.  Let $k_{m} = \lfloor \frac{d_{m}}{\dim X + 1} \rfloor$.
For large $m$, Proposition \ref{prop:totalbreaking} allows us to find a deformation of $s_m$ that breaks off $k_{m}$ rational curves through $k_{m}$ general points of $s(C')$.  Because $H$ is nef, at least one of these rational curves has $H$-degree at most 
 \begin{align*}
 \frac{H\cdot s_{m*}C'}{k_{m}} =  \frac{p^m(H \cdot i_*C')}{k_{m}} \leq  \frac{p^m(H \cdot i_*C')}{p^{m}(-K_{X} \cdot i_{*}C') - (g + 1)\dim X}(\dim X + 1).
 \end{align*}
 For large enough $m$ the floor of this upper bound is at most $(\dim X + 1)\frac{H \cdot i_*C'}{-K_X \cdot i_*C'}$.  This proves that a general point of $s(C')$ is contained in a rational curve whose $H$-degree satisfies the desired inequality.  Since the existence of such a rational curve through a point is a closed condition, this statement holds for every closed point in $s(C')$, and in particular for $x$.

The extension to ample $\mathbb{Q}$-Cartier divisors in characteristic $0$ uses the spreading out argument of \cite[Step 3 of proof of Theorem 5]{MM86}.  The extension to nef $\mathbb{R}$-Cartier divisors follows as in \cite[Steps 4 and 5 of proof of II.5.8 Theorem]{Kollar}. 
\end{proof}

%\nocite{*}
\bibliographystyle{alpha}
\bibliography{bandb}
\end{document}